\documentclass[english]{amsart}
\usepackage{comment}

\usepackage{MyCommA}
\usepackage{color}

\renewcommand{\restriction}{|}

\newcommand{\FS}{$\cM$-finite}

\newcommand{\Bi}{\mathfrak{B}}
\newcommand{\Pii}{\Pi}

\DeclareMathOperator{\linspan}{span}
\DeclareMathOperator{\Bl}{Bl}
\DeclareMathOperator*{\colim}{colim}
\usepackage{tikz-cd}

\title{Push-forward of smooth measures and strong Thom stratifications}
\author{Avraham Aizenbud, Nir Avni, and Shahar Carmeli}
\date{\today}

\begin{document}

\maketitle
\begin{abstract}
    We study the collection of measures obtained via push-forward along a map between smooth varieties over p-adic fields. We investigate when the stalks of this collection are finite-dimensional. We provide an algebro-geometric criterion ensuring this property. This criterion is formulated in terms of a canonical subvariety of the cotangent bundle of the source of the map.  
\end{abstract}

\section{Introduction}

\subsection{Main Result}
Let $\phi\colon X\to Y$ be a morphism of smooth algebraic varieties over a $p$-adic {(i.e., local, non-archimedean, and characteristic 0)} field $F$. 
{Let $\phi_*\cM_c(X(F))$ be the collection of measures on $Y(F)$ obtained from smooth compactly supported measures on $X(F)$ by push-forward.  $\phi_*\cM_c(X(F))$ is a module over the algebra $C_c^\infty(Y(F))$ of locally constant compacty supported functions on $Y(F)$ and  
$$C_c^\infty(Y(F))\cdot  \phi_*\cM_c(X(F))=\phi_*\cM_c(X(F)),$$ so we can c}onsider it as a sheaf\footnote{Beware that, unlike what the notation might suggest, this sheaf is \emph{not} defined as the pushforward of a sheaf on $X$.} on $Y$.

In this paper, we study the following question:
\begin{question}\label{que:main}
    When are the stalks of $\phi_*\cM_c(X{(F)})$ finite-dimensional?
\end{question}

{We partially answer \Cref{que:main} by giving} an algebro-geometric sufficient condition for this finite dimensionality. In order to formulate it, we make the following:

\begin{defn}[the scheme $\Bi_\phi$]\label{def:bi.phi}
For $\phi$ as above, we define the scheme
\[
\mdef{\Bi_\phi} := \Spec_X(\Sym(\Im(D\phi))),
\]
{where $D\phi$ denotes the differential of $\phi$. The} surjection $\cT_X\onto \im(D\phi)$ corresponds to a canonical embedding $\Bi_\phi \into T^*_X$ {that is compatible} with the structure map $\mdef{\Pii_{\Bi_\phi}}:\Bi_\phi\to X$.
\end{defn} 

\begin{introtheorem}[\Cref{thm:main}]\label{thm:intro.main}
Let $\phi\colon X\to Y$ be a morphism of smooth algebraic varieties defined over a $p$-adic field $F$. Assume that for every point $(x,v)\in \Bi_{\phi}$ we have 
\begin{equation}\label{eq:q.trans}   
\dim_{(x,v)}(\phi\circ \Pi_{\Bi_\phi})^{-1}(\phi(x)) \le \dim_x X.
\end{equation}
Then the stalks of $\phi_*\cM_c(X)$ are finite-dimensional. 
\end{introtheorem}

In order to prove this theorem, we show:

\begin{introtheorem}[\Cref{thm:trans_iff_thom}]\label{thm:intro.thom.crit}
Let $\phi\colon X\to Y$ be a morphism of smooth algebraic varieties over {a field of characteristic $0$}. Then $\phi$ satisfies {Inequality \eqref{eq:q.trans} for every $(x,v)\in \Bi_\phi$} if and only if $X$ and $Y$ admit compatible stratifications satisfying the following properties:
\begin{itemize}
    \item For every stratum $S\subseteq X$, the set $\phi(S)$ is a stratum, and the map $S\to \phi(S)$ is smooth.
    \item For every stratum $S\subseteq X$, a point $x\in S$, and a tangent vector $v\in T_xX$ which is tangent to $S$ and to the fiber of $\phi$, we can extend $v$ to a {regular} vector field in a neighborhood of $x$ which is tangent to the strata and to the fibers of $\phi$ at every point.  
\end{itemize}
\end{introtheorem}

\subsection{Background and Motivation}
{
\Cref{que:main} is a part of a
broader theme concerning the relationship between the algebraic–geometric properties of a morphism and the pushforward of smooth measures along this morphism. So far, such questions has been studied from a different perspective: the focus has been on the analytic properties of the pushforward of individual measures, rather than on their collective behavior \cite{AA,GH2,GHS,CH18,rei18}.
}

\Cref{que:main} is motivated by the following:
\begin{example}\label{ex:sha}
    Let $G$ be a reductive algebraic group and $\g$ its Lie algebra.
    Let $\g//G$ be the categorical quotient of $\g$ by the adjoint action of $G$.
    Let $p:\g\to \g//G$ be the quotient map. The theory of Shalika germs (\cite[Theorem 2.1.1]{Sha72}) implies that
    the stalk of $p_*\cM_c(\g(\Q_p))$ at $0$ is finite-dimensional, with an explicit basis, called Shalika germs, indexed by stable nilpotent orbits. 
\end{example}
A similar statement also holds for maps of the type $\phi:X\to X//G$ if $G$ acts with finitely many orbits on each fiber of $\phi$.
In this case, $\fB_\phi$ is contained in the union of the conormal bundles to the orbits of $G$ on $X$, so it easy to check  condition \eqref{eq:q.trans} in this case. Thus, \Cref{thm:intro.main} generalizes \Cref{ex:sha}.

\subsection{Outline of the Proofs}
We call maps that satisfy condition \eqref{eq:q.trans} of \Cref{thm:intro.main} {{quasi-transitive}}. We call stratifications as in \Cref{thm:intro.thom.crit} {{strong Thom stratifications}}.

\subsubsection{\Cref{thm:intro.thom.crit}}
Let $\phi:X\to Y$ be a morphism of smooth varieties (not necessarily quasi-transitive).

In order to show that $\phi$ admits a strong Thom stratification, we introduce a relaxation of the strong Thom condition. 
We call a stratification of $\phi$ {{vertically extendable}} 
if every tangent vector that is simultaneously tangent to both the fibers of $\phi$ the stratum of $X$ extends to a vector field that is tangent to the fibers (but not necessarily to the strata); see \Cref{def:vert.ex.strat} below.

We further introduce a dual notion: a stratification is called {{coarse}} if every vector field that is tangent to the fibers of $\phi$ is also tangent to the strata of $X$; see \Cref{def:coars} below.
By definition, a stratification that is both coarse and vertically extendable satisfies the strong Thom condition. 

We note that a tangent vector to $x\in X$ is perpendicular to $\Pi_{\fB_\phi}^{-1}(x)$ if and only if it can be extended to a vector field that is tangent to the fibers of $\phi$ (see \Cref{prop:fiber_of_B_phi} below). From this, we deduce the following: 

\begin{enumerate}[(i)]
\item A stratification is vertically extendable if and only if
\begin{equation*}\label{eq:intro.inc_2}
\fB_\phi\subset\bigcup_{y\in Y}\bigcup_{i} CN^X_{X_i \cap \phi^{-1}(y)},  
\end{equation*}
where $X_i$ are the strata in $X$, and $CN$ denotes the conormal bundle. See \Cref{lem:criterion_strong_thom_B} below.

\item 
A stratification is coarse if and only if the converse inclusion holds:
\begin{equation*}\label{eq:intro.inc_1}
\fB_\phi\supset\bigcup_{y\in Y}\bigcup_{i} CN^X_{X_i \cap \phi^{-1}(y)}.
\end{equation*}
See \Cref{lem:ver.ext.thom} below.
\end{enumerate}

In \Cref{sec:func.strat} (see \Cref{cor:exist.reg.strat}), we construct a stratification of $\phi$ that satisfies: 
\begin{itemize}
    \item The stratification is coarse.
    \item $\phi$ restricts to a smooth map between corresponding strata.
    \item The dimensions of the fibers of the sheaf $\Im D\phi$ are constant along the strata in $X$. 
\end{itemize}
We construct this stratification by a greedy algorithm.

Now assume in addition that $\phi$ is quasi-transitive.
We use a dimension argument to show that the inclusion in \eqref{eq:intro.inc_1} is an equality. We deduce by \eqref{eq:intro.inc_2} that the stratification is additionally vertically extendable, and hence satisfies the strong Thom condition. 

\subsubsection{\Cref{thm:intro.main}}

The existence of a Thom stratification can be thought of as an infinitesimal version of the existence of a group acting on $X$ (preserving $\phi$) with finitely many orbits in each fiber, where the group action is replaced by a sufficient supply of vector fields. Although these vector fields can be exponentiated locally in the analytic setting, we do not know how to use this fact directly to bound the dimensions of the desired stalks.

The Artin Approximation Theorem allows us replace the analytic flow of vector fields by an algebraic approximation of the formal flow, hence producing a Nisnevich neighborhood over which  
$\phi$ factors as a submersion followed by a map {from a lower dimensional space}.
Using this machinery, we prove a local structure theorem for Thom stratifications:

\begin{introtheorem}[\Cref{thm:structure.Thom}]\label{thm:intro.slice}
Let $\phi \colon X\to Y$ be a strongly Thom-stratified morphism, let $S\subset X$ be a closed stratum, and let $x\in S$. Assume that 
$$\dim(\phi^{-1}(\phi(x))\cap S)> 0.$$ 
Then, there exists a 
Nisnevich neighborhood $\tilde X$ of $x$ such that $\phi|_{\tilde X}$ factors as 
$$\phi|_{\tilde{X}}\colon \tilde X \xrightarrow{\pi} Z \xrightarrow{\tilde \phi} Y$$
such that
\begin{enumerate}
\item\label{thm:intro.slice:1} $\pi$ is a surjective submersion. 
\item\label{thm:intro.slice:2} $\dim(Z)< \dim(X)$.
\item\label{thm:intro.slice:3} $\tilde{\phi}$ is quasi-transitive. 
\end{enumerate}
\end{introtheorem}

\Cref{thm:intro.main} is proved by induction on the dimension of $X$ and the number of strata in a strong Thom stratification of $\phi$ (in lexicographical order).

If there is a stratum $S$ such that $\dim(\phi^{-1}(\phi(x))\cap S)=0$, then we deduce the theorem from the induction assumption applied to $\phi|_{X\smallsetminus S}$. Otherwise, we use \Cref{thm:intro.slice} and deduce the theorem from the induction assumption applied to $\tilde \phi$. 

{
\subsection{Questions}
We end with several questions related to geometric aspects of \Cref{que:main}. In order to formulate them we make the following definition:

\begin{definition} Let $K$ be a finitely generated field of characteristic zero, let $X,Y$ be smooth algebraic varieties over $K$ and let $\phi :X \rightarrow Y$ be a regular morphism. We say that $\phi$ is \tdef{\FS} if for every embedding $K \hookrightarrow F$ into a local field and every $y\in Y(F)$, the stalk $\phi_* \mathcal{M}_c(X(F))_y$ is finite-dimensional.
\end{definition} 

\Cref{thm:intro.main} implies that quasi-transitive morphisms (equivalently, by \Cref{thm:intro.thom.crit}, morphisms that admit strong Thom stratifications) are \FS.
We show in \Cref{exam:4lines} that the converse is not true.
However, our counterexample does satisfy Thom's Condition $A_\phi$ (see \cite[\S 5, Corollary 1]{Hir76}).

This gives rise to the following:
\begin{question}
    Is being \FS\ equivalent to the existence of a stratification that satisfies Thom's Condition $A_\phi$?
\end{question}

In view of \cite{Le76}, we ask
\begin{question}
    Is being \FS\ equivalent to having a theory of nearby cycles (see \cite[page 287]{Sab})?
\end{question}

Another question that one can ask about the \FS ness property is its stability under deformations. 

We say that a property $\mathcal{P}$ of morphisms is stable under smooth deformations if the following holds: For every commutative diagram 
\[
\xymatrix{X \ar[r]^{\varphi} \ar[dr]_{\eta} & Y \ar[d] ^{\psi}\\ & Z}
\]
of smooth varieties for which $\eta,\psi$ are smooth and every $x\in X$ such that the map $\varphi|_{\eta ^{-1} (\eta(x))}$ satisfies $\mathcal{P}$, there is a Zariski open neighborhood $U$ of $x$ such that for every $x'\in U$, the map $\varphi|_{\eta ^{-1} (\eta(x'))\cap U}$ satisfies $\mathcal{P}$.

\begin{question}
    Is the \FS ness property stable under deformations?
\end{question}

\begin{remark}
    Quasi-transitivity is not stable under deformation. See \Cref{rem:def} below.
\end{remark}

Other questions one can ask about the \FS ness and the quasi-transitivity properties concern their behavior under convolutions of morphisms and their relation with the notion of strength. We discuss these questions in \Cref{rem:con} below.
}
\subsection{Acknowledgments}
We thank David Kazhdan for suggesting \Cref{que:main}. 

A.A. was partially supported by ISF grant no. 1781/23, N.A. was partially supported by NSF grant DMS-2503233 and by Simons Foundation Award 1037177. S.C. was partially supported by ISF grant 4093/25, and would like to thank the Azrieli Foundation for their support through an Early Career Faculty Fellowship.
Both A.A. and N.A. were partially supported by BSF grant no. 2022193.
\section{Stratifications}

{By an algebraic variety we always mean a reduced and separated scheme of finite type over a field. For a scheme $X$ we denote the underlying topological space by \tdef{$|X|$}.}
{We introduce several standard notions related to stratifications of algebraic varieties.}

\begin{defn}[Stratification]
Let $X$ be a topological space. A \tdef{stratification} of $X$ is a continuous map $X\to P$ for a finite poset $P$ regarded as a topological space with respect to the order topology, {i.e., the topology with open sets the upwardly-closed ones.}
A \tdef{stratum} of a stratification is the preimage of a point in $P$. We call $P$ the indexing poset of the stratification. A \tdef{stratified space} is a topological space $X$ endowed with a stratification $p\colon X\to P$. When $p$ (and $P$) are clear from the context, we shall omit them from the notation and refer to $X$ as a stratified space. 
\end{defn}

\begin{rem}
Note that in this setup, the closure of each stratum is not necessarily a union of strata, and a stratum might be empty. 
\end{rem}

\begin{defn}
A \tdef{stratified map} of stratified topological spaces $X$ and $Y$ is a commutative square {of continuous maps}
\[
\xymatrix{
X\ar[r]\ar[d] & P\ar[d] \\
Y\ar[r] & Q
}.
\]
\end{defn}

{
\begin{defn}[Stratifications of algebraic varieties]
{A \tdef{stratification of an algebraic variety} $X$ is  a stratification of $|X|$.}
Let $(X,|X|\to P)$ and $(Y,|Y|\to Q)$ be stratified algebraic varieties.
A \tdef{stratified morphism} between $X$ and $Y$ is a morphism that induces a stratified map $|X|\to |Y|$. 

A \tdef{stratification of a morphism} $\phi\colon X\to Y$ is a choice of stratifications of $X$ and $Y$ such that $\phi$ is a stratified morphism with respect to these stratifications.
\end{defn}
}
\begin{defn}[Regular Stratification]
A stratification of an algebraic variety $X$ is called \tdef{regular} if all the strata are smooth algebraic subvarieties of $X$.
\end{defn}

\begin{defn}[Regular and Equistratified Morphisms]
A stratified morphism $\phi \colon X\to Y$ is called: 
\begin{enumerate}
\item  \tdef{regular} if the stratifications of $X$ and $Y$ are regular and for every stratum $S\subseteq X$ mapping to a stratum $T\subseteq Y$, the morphism $\phi|_S\colon S\to T$ is a submersion. 
\item \tdef{equistratified} if the map between the indexing posets is an isomorphism. 
\end{enumerate}
\end{defn}

\begin{rem}
In the case $F=\bC$, all the definitions above have evident analogs in the analytic category that we shall use freely. 
\end{rem}

\section{Quasi-transitive morphisms and strong Thom Stratifications}
In this section we introduce the notion of quasi-transitive morphisms (see \Cref{def:q.trans}) and the strong Thom condition on stratified morphisms (see \Cref{def:str.thom}). We also show that quasi-transitive morphisms admit a strong Thom stratification (see \Cref{thm:trans_iff_thom}).

{Let $F$ be a field.}
Let $\phi:X \rightarrow Y$ be a morphism of smooth varieties over $F$. Consider the differential $D\phi\colon \cT_X \to \phi^*\cT_Y$, where $\cT_X$ is the tangent sheaf for $X$ (whose sections are algebraic vector fields) and similarly for $Y$. We denote the kernel of $D\phi$ by $\cT_\phi$. 

\begin{defn}
A tangent vector $v\in T_x X$ is called \tdef{vertically extendable} if there is a Zariski open neighborhood $U\subset X$ of $x$ and a vector field $\xi \in \cT_\phi(U)$ such that $\xi(x)=v$.
\end{defn} 
{This notion is related to the following subvariety of the contangent bundle of $X$.} 
\begin{defn}[the scheme $\Bi_\phi$]
For $\phi$ as above, we define the scheme
\[
\mdef{\Bi_\phi} := \Spec_X(\Sym(\Im(D\phi))).
\]
{The} surjection $\cT_X\onto \im(D\phi)$ corresponds to a canonical embedding $\Bi_\phi \into {T^*X=\Spec_X(\Sym(\cT_X))}$, compatibly with the structure map $\Pii_{\Bi_\phi}:\Bi_\phi\to X$.
\end{defn} 
{The subscheme $\Bi_\phi\subseteq T^*X$ determines the vertically extendable vectors as follows:}

\begin{proposition} \label{prop:fiber_of_B_phi}
{Let $\phi:X\to Y$ be a morphism of smooth algebraic varieties.}
The fiber of $\Bi_\phi$ over $x\in X$, considered as a subspace of $T^*_x X$, is the annihilator of the space of vertically extendable vectors.
\end{proposition} 

\begin{proof}
Since the formations of symmetric algebra and relative spectrum {are} compatible with base change, the fiber of $\Bi_\phi$ over $x\in X$ {is}
\begin{equation}\label{eq:B_phi_fiber}
((\im D\phi)|_x)^* \cong \Spec_x(\Sym(\im D\phi|_x)) \subseteq \Spec_x(\Sym(T_x X)) \cong T^*_xX. 
\end{equation}
Consider the exact sequence 
\[
0 \to \cT_\phi \xrightarrow{i} \cT_X \oto{\beta} \im(D\phi) \to 0.
\]
Taking the fiber at $x$ we get an exact sequence 
\[
\cT_\phi|_x \oto{i_x} T_x X \oto{\beta_x} \im(D\phi)|_x \to 0.
\]
Hence, the image of $i_x$ is the kernel $K$ of the map $T_x X \oto{\beta_x} \im(D\phi)|_x$. 
By definition, the image of $i_x$ is exactly the collection of vertically extendable vectors. Therefore, the image of the dual map $\im(D\phi)|_x^* \oto{\beta_x^*} T_x X^*$ is the annihilator of the space of vertically extendable vectors, which by \eqref{eq:B_phi_fiber} is the fiber of $\Bi_\phi$ over $x$.  
\end{proof} 

{Thus, informally, the abundance of vertically extendible vectors is linked to the ``smallness'' of $\Bi_\phi$. This motivates the following definition.}
\begin{defn}[Quasi-transitive morphism]\label{def:q.trans}
Let $\phi \colon X\to Y$ be a morphism of smooth algebraic varieties. 
We say that $\phi$ is \tdef{quasi-transitive} if for every $(x,v)\in \Bi_\phi$ we have 
\[
\dim_{(x,v)}\Big( (\phi \circ \Pii_{\Bi_\phi})^{-1}(\phi(x)) \Big) \leq \dim_x X.
\]
\end{defn}

{
Next, we discuss related notions for \emph{stratifications} of a morphism. Eventually, we will see that quasi-transitivity is equivalent to the existence of a strong Thom stratification as defined below. It will be useful to have also the following relaxed notion:}

\begin{defn}[Vertically Extendable Stratified Morphism]\label{def:vert.ex.strat} 
Let $\phi \colon X\to Y$ be a regularly stratified morphism of smooth algebraic varieties.
 We say that the stratification of $\phi$ is \tdef{vertically extendable} if for every $x\in X$ belonging to a stratum $S$ and every $v\in T_{x}{(}\phi^{-1}(\phi(x)) \cap S{)}$ the vector $v$ is vertically extendable. In this case, we refer to $\phi$ as a vertically extendable (stratified) morphism.  
\end{defn}

\begin{defn}[Strong Thom Stratification]\label{def:str.thom} 
Let $\phi \colon X\to Y$ be a regularly stratified morphism of smooth algebraic varieties.
 We say that the stratification of $\phi$ satisfies the \tdef{strong Thom condition} if for every $x\in X_i$ and every $v\in T_x( \phi^{-1}(\phi(x)) \cap X_i)$ the vector $v$ extends to a vertical vector field $\tilde{v}\in \cT_\phi$ defined in a neighborhood $U$ of $x$ such that, for all $x'\in U$ belonging to a stratum $X_j$, we still have 
\[
\tilde{v}(x') \in T_{x'}(X_j). 
\]
In this situation, we will also refer to $\phi$ as a \tdef{strongly Thom-stratified} morphism. 
\end{defn}

Note that, in particular, every strongly Thom stratified morphism is vertically extendable. 
{Now we use \Cref{prop:fiber_of_B_phi} in order to  characterize the vertically extendable stratifications in terms of the variety $\Bi_\phi$.}

\begin{lemma}\label{lem:criterion_strong_thom_B}
Let $\phi \colon X\to Y$ be a 
{regular} stratified morphism of smooth algebraic varieties.
Then $\phi$ is vertically extendable if and only if for all $y\in Y$, 
\[
{\left|(\phi \circ \Pii_{\Bi_\phi})^{-1}(y) \right| \subseteq \bigcup_{i\in I} \left|\CN^X_{X_i \cap \phi^{-1}(y)}\right|}.  
\]
\end{lemma}

\begin{proof}
Let $x\in \phi^{-1}(y) \cap X_i$. By \Cref{prop:fiber_of_B_phi}, the fiber of the left-hand side over $x$ is the annihilator of the space of vertically extendable vectors in $T_x X$. On the other hand, the fiber of the right-hand side is the annihilator of $T_x( \phi^{-1}(y) \cap X_i)$. This implies the assertion.
\end{proof}

{As explained above, the strong Thom condition is  stronger than being vertically extendable. To bridge the gap between these notions, we use the following auxiliary one.}    
\begin{defn}\label{def:coars}
A regular stratification of a morphism $\phi \colon X\to Y$ is called \tdef{coarse} if every vector field $v\in \cT_\phi$ defined on an open set $U\subseteq X$ is tangent to the strata of $X$ at any point of $U$. In this case, we say that $\phi$ is \tdef{coarsely stratified}.  
\end{defn}

The next lemma now follows immediately from the definitions:
\begin{lem}\label{lem:ver.ext.thom}
Let $\phi \colon X\to Y$ be a regular stratified morphism. If $\phi$ is coarse and vertically extendable then it is strongly Thom stratified. 
\end{lem}

{The characterization of $\Bi_\phi$ from \Cref{prop:fiber_of_B_phi} gives us the following criterion for a morphism being coarsely stratified:}
\begin{cor}\label{cor:criterion_cours}
Let $\phi \colon X\to Y$ be a 
{regular} stratified morphism of smooth algebraic varieties.
Then $\phi$ is coarsely stratified if and only if for all $y\in Y$, 
{
\begin{equation}    \label{eq:incl}
\left|(\phi \circ \Pii_{\Bi_\phi})^{-1}(y) \right|\supseteq \bigcup_{i\in I} \left|\CN^X_{X_i \cap \phi^{-1}(y)}.\right|  \end{equation}
}
\end{cor}

\begin{proof}
As in {the proof of} \Cref{lem:criterion_strong_thom_B}, Condition \eqref{eq:incl} is equivalent to the condition
\begin{equation} \label{eq:incl2}
\text{For $x\in X_i$ and $v\in T_xX$, if $v$ is vertically extendable, then $v$ is tangent to $X_i$.}
\end{equation}

Assume that $\phi$ is coarsely stratified{, let $x\in X$,} and let $v\in T_xX$ be a vertically extendable vector. Let $\tilde v$ be an extension of $v$ to a section of $\cT_\phi$ defined in a neighborhood of $x$. Since $\phi$ is coarsely stratified, $\tilde v$ is tangent to the strata of $X$. Therefore $v$ is tangent to the stratum of $x$ {and Condition \eqref{eq:incl2} holds}.

Conversely, if Condition \eqref{eq:incl2} holds and $\tilde v$ is a section of $\cT_\phi$, then $\tilde v(x)$ is vertically extendable and hence tangent to the stratum of $x$.
\end{proof}

\subsection{Construction of a Functorial Stratification} \label{sec:func.strat}

In this section, for every morphism $\phi \colon X\to Y$ of smooth varieties we construct a \emph{functorial} regular stratification on $\phi$. This construction will have various desirable properties, such as being coarse (see \Cref{cor:exist.reg.strat}). 

We will work in parallel {both in} the settings of algebraic varieties and {in} the setting of complex analytic varieties. Note that the notion of a coarse stratification described above extends to the setting of complex analytic varieties.

\begin{defn} \label{def:M_reg}
Let $\phi \colon X\to Y$ be a morphism of algebraic  varieties (or complex analytic varieties) and let $M$ be a coherent sheaf of $\mathcal{O}_X$-modules. A regular stratification of $\phi$ is said to be \tdef{$M$-regular} if, for every stratum $X_i\subseteq X$, the restriction $M|_{X_i}$ is locally free.
\end{defn}

\begin{thm}\label{thm:nat.strat}
{Let $F$ be a field of characteristic $0$}. There are: 
\begin{enumerate}
\item An assignment $\fS_F$ that assigns to a {pair} $(\phi\colon X\to Y,M)$ as in \Cref{def:M_reg}, defined over $F$, an $M$-regular stratification of $\phi$. 
\item A corresponding assignment $\fS_\an$ that assigns to a triple $(\phi\colon X\to Y,M)$ as above where $X,Y,\phi$ and $M$ are analytic over $\bC$ an $M$-regular stratification of $\phi$.
\end{enumerate}
such that
\begin{enumerate}[(i)]
    \item \label{thm:nat.strat:3}    
    $\Ext_F^{F'} \circ \fS_F=\fS_{F'}  \circ  \Ext_F^{F'}$ where  $\Ext_{F}^{F'}$ denotes  extension of scalars {from $F$ to $F'$}.
    \item \label{thm:nat.strat:ii}$\An\circ \fS_{\bC} = \fS_\an \circ \An$ where $\An$ is the complex analytification functor.
    \item\label{thm:nat.strat:iii} For a {pair} $(\phi\colon X\to Y,M)$ in the analytic category as above and a pair of open embeddings $i_{1,2}:U\to X$ such that $\phi \circ i_1=\phi \circ i_2$ and $i_1^*M \simeq i_2^*M $, 
    we have 
    $$i_1^{-1}(\fS_\an(\phi,M))= i_2^{-1}(\fS_\an(\phi,M)).$$
\end{enumerate}
\end{thm}

{For the constructions of $\fS_F$ and $\fS_\an$ we shall need several auxiliary stratifications.}
\begin{definition}
$ $
\begin{enumerate}
    \item Let $X$ be an algebraic or an analytic variety.  We define a stratification $\Reg(X)$ on $X$ in the following recursive way: 
    \begin{itemize}
        \item If $X$ is empty then so is $\Reg(X)$.
        \item Consider the stratification $\Reg(X^{sing})\colon X^{sing}\to P$. Let $Q=\{0\}\sqcup P$ where $0$ is the maximal element, and define $\Reg(X)\colon X\to Q$ by 
        \[
        \Reg(X)(X^{sm}):=\{0\}\quad \text{and} \quad \Reg(X)|_{X^{sing}}:=\Reg(X^{sing}).
        \]
    \end{itemize}
    
    \item Let $(X,p)$ be a stratified algebraic or analytic variety.  We denote by $\Reg(X,p)$ the regular stratification on $X$ obtained in the following way. Let $X_i$ be the strata of $X$. Consider the stratification $\Reg(X_i)\colon X_i\to P_i$. Let $P:=\bigsqcup P_i$ ordered such that for $i>j$ each element of $P_i$ is greater than each element of $P_j$ and the order on each $P_i$ is unchanged. Define $\Reg(X,p)\colon X\to P$ by 
    \[
    \Reg(X,p)|_{X_i}=\Reg(X_i).
    \]
\end{enumerate}
\end{definition}

{We next explain how to glue stratifications of a morphism along an open-closed decomposition of its target.}
\begin{definition}[Target Gluing]
    Let $\phi:X\to Y$ be a morphism of algebraic or analytic varieties and let $|Y|=|Z|\cup |U|$ be a decomposition of $Y$ as a {disjoint union of an open subvariety  $U$ and a closed subvariety $Z$}. Let $\phi_U:\phi^{-1}(U)\to U$ and $\phi_Z:\phi^{-1}(Z)\to Z$ be obtained as restrictions of $\phi$. Assume that
    $$
    \begin{tikzcd}
        \phi^{-1}(U)\rar["s_U"]\dar["\phi_U"] & S_U \dar["\alpha_{\phi_U}"]\\
        U \rar["t_U"] & T_U 
    \end{tikzcd}
    $$
 and     
    $$
    \begin{tikzcd}
        \phi^{-1}(Z)\rar["s_Z"]\dar["\phi_Z"] & S_Z \dar["\alpha_{\phi_Z}"]\\
        Z \rar["t_Z"] & T_Z 
    \end{tikzcd}
    $$
are stratifications of $\phi_U$ and $\phi_Z$.
    
    Define \tdef{$\phi_U*_Y\phi_Z$} to be the stratified map obtained from the following stratification on $\phi$:
    \begin{itemize}
        \item 
We set $T:=T_U\cup T_Z$ and extend the order relations on $T_U$ and $T_Z$ such that each element of $T_U$ is larger than each element of $T_Z$.
        \item We set $S:=S_U\cup S_Z$ and define the order similarly.
        \item We define  $s:X\to S$ by $$s(x)=\begin{cases}
			s_U(x), & \text{if $x\in \phi^{-1}(U)$}\\
			s_Z(x), & \text{if $x\in \phi^{-1}(Z)$}
		 \end{cases}$$
    \item We define  $t:Y\to T$ and $\alpha_\phi:S\to T$ similarly.    
    \end{itemize}
This gives us the {stratified morphism $\phi_{U}\ast_Y \phi_Z$:} 
    $$
    \begin{tikzcd}
        X\rar["s"]\dar["\phi"] & S \dar["\alpha_{\phi}"]\\
        Y \rar["t"] & T 
    \end{tikzcd}
    $$ 
\end{definition}

\begin{definition}[Source Gluing]
    Let $\phi:X\to Y$ be a morphism of algebraic or analytic varieties and let $X= U \cup Z$ be a decomposition of $X$ into an open and a closed set. Let $\phi_U:=\phi|_U:U\to Y$ and $\phi_Z:=\phi|_Z:Z\to Y$. Assume that
    $$
    \begin{tikzcd}
        U\rar["s_U"]\dar["\phi_U"] & S_U \dar["\alpha_{\phi_U}"]\\
        Y \rar["t_U"] & T_U 
    \end{tikzcd}
    $$
 and     
    $$
    \begin{tikzcd}
        Z\rar["s_Z"]\dar["\phi_Z"] & S_Z \dar["\alpha_{\phi_Z}"]\\
        Y \rar["t_Z"] & T_Z 
    \end{tikzcd}
    $$
    are stratifications on $\phi_U$ and $\phi_Z$. 
    Define \tdef{$\phi_U*_X\phi_Z$} to be the stratified map obtained from the following stratification on $\phi$:
    \begin{itemize}
        \item 
We set $T_0:=T_U \times T_Z$ and consider the coordinate-wise order relation on it.
\item Consider the stratification $t\colon Y\to T_0$ given by $t(y)=(t_U(y),t_Z(y))$. 
\item Let $(Y\to T) := \Reg(Y,t)$ be its image under the construction $Reg$.  
\item We set $S^0:=S_U\cup S_Z$ and extend the order relations on $S_U$ and $S_Z$ such that each element of $S_U$ is larger than each element of $S_Z$.
\item We set $S:=S^0 \times T$ and consider the coordinate-wise partial order on it.
\item We define  $s:X\to S$ by $$s(x)=\begin{cases}
			(s_U(x),t(\phi(x))), & \text{if $x\in U$}\\
			(s_Z(x),t(\phi(x))), & \text{if $x\in Z$}
		 \end{cases}$$
         \item We define $\alpha_\phi:S\to T$ to be the projection on the second coordinate.
    \end{itemize}
We define $\phi_U*_X\phi_Z$ to be the stratified map as in the following diagram 
\[
    \begin{tikzcd}
        X\rar["s"]\dar["\phi"] & S \dar["\alpha_{\phi}"]\\
        Y \rar["t"] & T 
    \end{tikzcd}
\]   
\end{definition}

The following lemma is straightforward {and we omit the proof}:
\begin{lem}\label{cor:glue.fanct}
    The procedures of source gluing and target gluing are functorial with respect to open embeddings on the source. Namely,
    let $i:X'\to X$ be an open embedding and $\phi:X\to Y$ be a morphism of (algebraic or analytic) varieties. 
    \begin{enumerate}
        \item Let $Y=Z\cup U$ be a decomposition of $Y$ into an open and a closed set. Let $\phi_U:\phi^{-1}(U)\to U$ and $\phi_Z:\phi^{-1}(Z)\to Z$ be the restrictions of $\phi$. Equip both $\phi_U$ and $\phi_Z$ with stratifications. Define similarly $i_U:i^{-1}(\phi^{-1}(U))\to \phi^{-1}(U)$ and $i_Z:i^{-1}(\phi^{-1}(Z))\to \phi^{-1}(Z)$ and put stratifications on their sources that make $i_U$ and $i_Z$ equistratified. Then $i$ is an equistratified map with respect to the stratifications $\phi_U \circ  i_U*_Y \phi_Z \circ  i_Z$ and $\phi_U *_Y \phi_Z$.        
\item Let $X= U \cup Z$ be a decomposition of $X$ into an open and a closed set, and equip both $\phi|_U$ and $\phi|_Z$ with a stratification.
Consider the stratified map $\phi|_U \ast_X \phi|_Z$. We can pull back the resulting stratification of $X$ along $i$ to get a stratified structure on $i\circ (\phi|_U \ast_X \phi|_Z)$.  On the other hand, we have the stratified map $(\phi|_U \circ i_U) \ast_{X'}(\phi|_Z\circ i_Z)$. Then, 
\[
i\circ (\phi|_U \ast_X \phi|_Z) = (\phi|_U \circ i_U) \ast_{X'}(\phi|_Z\circ i_Z)
\]
as stratified maps $X' \to Y$. 
    \end{enumerate}
\end{lem}

\begin{proof}[Proof of Theorem \ref{thm:nat.strat}]
We will construct the assignments $\fS_F$ and $\fS_{an}$ in the same way. In what follows we fix one case and denote the corresponding {assignment} simply by $\fS$. 
Our construction is recursive with respect to the partial order on the {pair} $(\dim X, \dim Y)$. 
We proceed by analyzing several cases. The construction in each case may depend on pairs with smaller dimensions or on previous cases.

For any {pair} $(\phi\colon X\to Y,M)$, we will use the construction in the first case that is applicable to this {pair}:
\begin{enumerate}[{Case} 1:]
    \item $\phi$ is a submersion of smooth varieties and $M$ is locally free.\\
    In this case we define $\fS(\phi,M)$ to be the trivial stratification.
    \item $\phi$ is a submersion of smooth varieties.\\
    Let {$U \subseteq X$ be the locus of points over which the fibers of $M$ are of smallest dimension, i.e.,}
    \[U=\{x\in X\mid \dim M|_x=\min_{y\in X_x} \dim M|_y\},
    \]
    where $X_x$ is the component of $X$ that contains $x$. Set $X':=X\smallsetminus U$, $\phi':=\phi|_{X'}$, and $M':=M|_{X'}$. 
    Define 
    $$\fS(\phi,M):=\fS(\phi',M')*_X \fS(\phi|_U,M|_U).$$
    \item $X,Y$ are smooth.\\
    Let $U\subset Y$ be the locus of regular values of $\phi$. 
    Let $Y':=Y\smallsetminus U$. Let $\phi':\phi^{-1}(Y')\to Y'$ and 
    $\phi_U:\phi^{-1}(U)\to U$ be the maps obtained from the restriction of $\phi$.  Let    
    $M':=M|_{\phi^{-1}(Y')}$. 
    {Since $\mathrm{char}(F)=0$, we have $\dim(Y')<\dim(Y)$. Hence, we can apply $\fS$ to the pair $(\phi',M')$.}
    Define $$\fS(\phi,M):=\fS(\phi',M')*_Y \fS(\phi_U,M|_U).$$
    \item $Y$ is smooth.\\
        Let $U\subset X$ be the smooth locus of $X$. Let $X':=X\smallsetminus U$, $\phi':=\phi|_{X'}$, and $M':=M|_{X'}$. 
    Define $$\fS(\phi,M):=\fS(\phi',M')*_X \fS(\phi|_U,M|_U).$$
    \item General case.\\
    Let $U\subset Y$ be the smooth locus of $Y$. 
    Let $Y':=Y\smallsetminus U$, $\phi':\phi^{-1}(Y')\to Y'$, and 
    $\phi_U:\phi^{-1}(U)\to U$ be the restriction of $\phi$. Finally, let   
    $M':=M|_{\phi^{-1}(Y')}$. 
    Define $$\fS(\phi,M):=\fS(\phi',M')*_Y \fS(\phi_U,M|_U).$$
\end{enumerate}
{It remains} to show that conditions (i)–(iii) in the formulation are satisfied. This is done by induction on the {pair} $(\dim X,\dim Y)$, so we only need to check that the conditions are satisfied in each of the cases 1–5 above, provided that they are satisfied for morphisms for which $\fS$ {have been  already constructed}. {Since in each of the cases the stratification $\fS(\phi,M)$ is constructed via source or target gluing of previously constructed stratifications,} this follows  from 
\Cref{cor:glue.fanct}. 
\end{proof}

\begin{remark}
    Note that the functoriality in \Cref{cor:glue.fanct} is stronger than Condition \eqref{thm:nat.strat:iii} in \Cref{thm:nat.strat}. The reason is that while each case of the construction in \Cref{thm:nat.strat} is as functorial as in \Cref{cor:glue.fanct}, the 
    choice of the case is not functorial.  
\end{remark}
From now on we fix assignments $\fS_F$ and $\fS_{an}$ as in \Cref{thm:nat.strat}. 
\subsection{Existence of coarse regular stratification}
{Using the assignment $\fS$, we turn to show that every morphism $\phi\colon X\to Y$ admits a stratification that is simultaneously coarse and regular with respect to the sheaf $\im(D\phi)$.
For quasi-transitive morphisms, such stratifications will be shown later to satisfy the strong Thom condition.}
\begin{prop}
Let $\phi \colon X\to Y$ be a morphism of smooth complex {analytic} varieties. 

Then $\fS_{an}(\phi,\Im(D\phi))$ is coarse.
\end{prop}
\begin{proof} {Let $x\in X$, let $V$ be a neighborhood of $x$, and let $\xi \in \cT_\phi(V)$.} By the existence and uniqueness of solutions of analytic ODE, we have:
\begin{itemize}
    \item a neighborhood $U {\subseteq V}$ of $x$;
    \item a disk $D\subset \C$ around $0$;
    \item a map $\nu:U\times D\to X$ 
\end{itemize}
such that $\frac{d}{dt}\nu(y,t)\big|_{t=0}=\xi(x)$.
Define $\nu_t(y):=\nu(y,t)$. By \Cref{thm:nat.strat}\eqref{thm:nat.strat:iii} for any stratum $X_i\subset X$ the variety $\nu_t^*(X_i)$ does not depend on $t$. This proves the assertion.
\end{proof}

\begin{corollary}\label{cor:exist.reg.strat}
   Let $\phi \colon X\to Y$ be a morphism of algebraic varieties over {a field $F$ of characteristic $0$}. There is a coarse $\Im(D\phi)$-regular 
   stratification of $\phi$.
\end{corollary}

\begin{proof}
We take the stratification $\fS_F(\phi,\Im(D\phi))$. Let $k\subset F$ be a $\Q$-finitely generated subfield over which we can define $\phi$. By \Cref{thm:nat.strat}\eqref{thm:nat.strat:3} it is enough to prove the assertion for $F=k$. Embed $k$ into $\C$. So again by \Cref{thm:nat.strat}\eqref{thm:nat.strat:3} it is enough to prove the assertion for $F=\C$. The assertion follows now from the previous corollary using \Cref{thm:nat.strat}\eqref{thm:nat.strat:ii}.
\end{proof}

\subsection{Quasi-transitive morphisms admit strong Thom Stratifications}
We now turn to the main goal of this section. {Namely, we show that quasi-transitive morphisms admit strong Thom stratificaitons}. We start with the following: 

\begin{prop} \label{prop:quasi_trans_is_thom}
Let $\phi\colon X\to Y$ be a quasi-transitive morphism. Every stratification of $\phi$ which is coarse and $\Im(D\phi)$-regular is vertically extendable. 
\end{prop}

\begin{proof}
Let $\{X_i\}_{i\in I}$ be the strata in $X$. Since the stratification is coarse, every vector field in $\cT_\phi$ is tangent to $X_i$ at every point of it. 

By \Cref{cor:criterion_cours}, it follows that 
{
$$\left|(\phi \circ \Pii_{\Bi_\phi})^{-1}(y) \right|\supseteq \bigcup_{i\in I} \left|\CN^X_{X_i \cap \phi^{-1}(y)}\right|$$
}
for all $y\in Y$.
By \Cref{lem:criterion_strong_thom_B} it suffices to show that this inclusion is an equality. 

Let $\Pi\colon T^*X \to X$ be the projection. 
Since {$|X| = \bigcup_i |X_i|$}, it suffices to show that, for every $j$,  
{
\begin{equation} \label{eq:trans_iff_thom_1}
\left|(\phi\circ \Pii_{\Bi_\phi})^{-1}(y) \right|\cap \left| \Pi^{-1}(X_j)\right| = \bigcup_{i\in I} \left|\CN^X_{X_i \cap \phi^{-1}(y)}\right| \cap \left|\Pi^{-1}(X_j)\right|.
\end{equation}
}
Note that the right-hand side equals {$\left|\CN^X_{X_j \cap \phi^{-1}(y)}\right|$, which is the underlying space of} a vector bundle over $X_j \cap \phi ^{-1} (y)$. Since the stratification is $\Im(D\phi)$-regular, the left-hand side is also {the underlying space of }a vector bundle over $X_j \cap \phi^{-1}(y)$. 
By the quasi-transitivity assumption, for every $(x,v)\in \CN^X_{X_j \cap \phi^{-1}(y)}$
we have that
$$
    \dim_{(x,v)} \big((\phi\circ \Pii_{\Bi_\phi})^{-1}(y) \cap \Pi^{-1}(X_j)\big) \le \dim_{(x,v)} \big((\phi\circ \Pii_{\Bi_\phi})^{-1}(y)\big)=\dim_x(X)=\dim_{(x,v)}\big(\CN^X_{X_j \cap \phi^{-1}(y)}\big).
$$

Thus, both sides in \eqref{eq:trans_iff_thom_1} 
are {the underlying spaces of }vector bundles over $X_j\cap \phi^{-1}(y)$ of the same rank. Also, 
the RHS is contained in the LHS. Therefore
they are equal, as required.
\end{proof}

\begin{thm}\label{thm:trans_iff_thom}
Let $\phi\colon X\to Y$ be a morphism of smooth algebraic varieties {over a field of characteristic $0$}. Then, the following are equivalent:
\begin{enumerate}
    \item $\phi$ is quasi-transitive; 
    \item $\phi$ admits a strong Thom stratification. 
\end{enumerate}
\end{thm}

\begin{proof} $ $
\begin{itemize}
    \item[(2)$\Rightarrow$(1)] 
    Follows immediately from \Cref{lem:criterion_strong_thom_B}.
    \item[(1)$\Rightarrow$(2)] 
    By \Cref{cor:exist.reg.strat}, $\phi$ admits a coarse, $\Im(D\phi)$-regular stratification. By  \Cref{prop:quasi_trans_is_thom}, such a stratification is vertically extendable. Hence, by  \Cref{lem:ver.ext.thom}, this stratification satisfies the strong Thom condition, as required.
\end{itemize}
\end{proof}
\section{Local structure theorem for strongly Thom stratified morphisms}

In this section we state a precise version of \Cref{thm:intro.slice} (see \Cref{thm:structure.Thom}) and prove it.
\subsection{Idea of the proof}
In \Cref{ssec:des}, we show that quasi-transitivity is a local property in the smooth topology. Thus part \eqref{thm:intro.slice:3} of \Cref{thm:intro.slice} follows from part \eqref{thm:intro.slice:1}.
Therefore, the conclusion of \Cref{thm:intro.slice} can be restated as the existence of a map $\tilde X\to X\times_Y Z$ such that $\tilde X\to X$ is a Nisnevich neighborhood of $x$ and $\tilde X\to Z$ is a submersion.

By Artin’s approximation theorem it is enough to construct an analogous map  
$$\hat X_x\to X\times_Y Z$$
from the completion of $X$ at $x$.

We take $Z$ to be a transversal to $\phi^{-1}(\phi(x))\cap S$.

Let $V:=T_x(\phi^{-1}(\phi(x))\cap S)$. We have a map $$\hat X_x\simeq \widehat {Z\times V}_{(x,0)} \to \hat Z_x \times  \hat V_0.$$
This gives us a map $\psi_1:\hat X_x\to Z$.

We use the strong Thom property to extend the inclusion $V\to T_xX$ to a linear embedding $$V\to 
\cT_{\phi|_{\hat X_x}}:=\Ker D(\phi|_{{{\hat X_x}}}).$$ 
Using the existence and uniqueness theorem for formal ODEs this defines a map 
$$ \hat X_x \times \hat V_0 \to {\hat X}_x.$$
Define a map $\psi_2:\hat X_x \to X$ to be the composition 
$$\hat X_x \to  
\hat Z_x \times  \hat V_0 \hookrightarrow \hat X_x \times  \hat V_0 \to \hat X_x
\hookrightarrow  X.$$
The maps $\psi_1$ and $\psi_2$ give the desired map $\hat X_x\to X\times_Y Z$.





\subsection{Descent for Quasi-Transitive Morphisms}\label{ssec:des}

Recall that a correspondence between two algebraic varieties $X$ and $Y$ is a diagram of the form $X \from Z \to Y$.
Let $\psi\colon X\to Y$ be a morphism. Then, $\psi$ defines a correspondence 
\[
\psi^\dagger\colon T^*Y  \xleftarrow{\pi} T^*Y\times_Y X \xrightarrow{D\psi^*} T^*X
\]
with a dual correspondence 
\[
\psi_\dagger\colon T^*X \xleftarrow{D\psi^*}  T^*Y\times_Y X   \xrightarrow{\pi} T^*Y.  
\]
We can apply $\psi^\dagger$ to subvarieties of $T^*Y$ to get (constructible) subsets of $T^*X$. Note that when $\psi$ is a submersion, the right leg of the correspondence is a closed immersion and $\psi^\dagger$ carries subvarieties to subvarieties.   
We have the following straightforward relation between these operations.
\begin{lem} \label{lem:dagger_dim}
Let $\psi \colon X\to Y$ be a submersion {of algebraic varieties} and let $(x,v)\in T^*X$ and $(y,u)\in T^*Y$ be corresponding points under the correspondence $\psi^\dagger$. 
Then, for every $W\subseteq T^*Y$ such that $(y,u)\in W$, we have 
\[
\dim_{(x,v)}(\psi^\dagger W) = \dim_{(y,u)}(W) + \dim_{x} X - \dim_y Y.
\] 
\end{lem}

\begin{proof} 

Since $\psi$ is a submersion and $\pi$ is a pullback of $\psi$ we have:
\[
\dim_{(x,u)}(\pi^{-1}(W)) = \dim_{(y,u)}(W) + \dim_{(x,u)}T^*Y\times_Y X - \dim_{(y,u)} T^*Y = \dim_{(y,u)}(W) + \dim_x X - \dim_y Y. 
\]
Moreover, $D\psi^*$ is a closed immersion and hence 
$$\psi^\dagger W = D\psi^*(\pi^{-1}(W))\cong \pi^{-1} (W)$$
so the result follows.
\end{proof}

\begin{prop}[Pullback for $\Bi$]\label{prop:pullback_for_Bi}
Let $X \xrightarrow{\psi} Y \xrightarrow{\phi} Z$ be morphisms {of algebraic varieties}. If $\psi$ is a submersion then 
\[
\Bi_{\psi\circ  \phi} = \psi^\dagger\Bi_\phi.
\]
\end{prop}

\begin{proof}
First, observe that $T^*Y \cong \Spec_Y(\Sym(\cT_Y))$ and $T^*X\times_X Y \cong \Spec_X(\Sym(\psi^*\cT_X))$. 
Thus, every surjection $\cT_X \onto \fF$ gives a closed subvariety $\Spec_X(\Sym(\fF))\subseteq T^*X$, and similarly for $T^*Y$ and $T^*Y\times_Y X$. 
We shall describe how the operations $\pi^{-1}$ and $D\psi^*$ operate on subvarieties of the form $\Spec(\Sym(\fF))$ for subvarieties associated with such quotient sheaves.

\begin{itemize}
\item For $\xi \colon \cT_Y\onto \fF$, by the compatibility of the relative spectrum construction with base change we have $\pi^{-1}(\Spec_Y(\Sym(\fF))) = \Spec_X(\Sym(\psi^*\fF))$, regarded as a subvariety of $T^*Y\times_Y X = \Spec(\Sym(\psi^*\cT_Y))$ via the surjection $\psi^*\xi\colon \psi^*\cT_Y\onto \psi^*\fF$.

\item For $\nu \colon \psi^*\cT_Y\onto \fF$, since $D\psi^*$ is a closed embedding corresponding to the surjection $D\psi \colon \cT_X \to \psi^*\cT_Y$,  
$D\psi^*(\Spec_X(\Sym(\fF)))$ is the subset of $T^*X$ corresponding to the composite surjection
\[
\cT_X \xrightarrow{D\psi} \psi^*\cT_Y \xrightarrow{\nu} \fF
\]
via the construction $\Spec_X(\Sym(-))$.
\end{itemize}

Applying these two observations, by the definition of $\Bi_{\phi}$ and $\Bi_{\psi\circ \phi}$, it remains to identify the surjection 
\[
\cT_X \xrightarrow{D\psi} \psi^*\cT_Y \onto \psi^*\im(D\phi) 
\]
with the surjection 
\[
\cT_X \onto \im(D(\psi\circ \phi)).  
\]
Consider the diagram 
\[
\xymatrix{
\cT_X\ar@{->>}[d]\ar_{D(\phi\circ \psi)}[rrd] \ar^{D\psi}@{->>}[r] & \psi^*\cT_Y \ar@{->>}[r] \ar^{\psi^*D\phi}[rd] & \psi^*\im(D\phi)\ar[d] \\ 
\im(D(\phi\circ \psi))\ar[rr] &  & \psi^*\phi^*\cT_Z \\ 
}
\]
It commutes because {of the chain rule} $\psi^*D\phi\circ D\psi = D(\phi\circ \psi)$. The right vertical map is an injection since $\psi$ is flat and the lower horizontal map is injective by the definition of image. The upper left horizontal morphism is a surjection since $\psi$ is a submersion and the upper right horizontal map is a surjection since $\psi^*$ is right exact. Finally, the left vertical morphism is a surjection by the definition of image. Since the decomposition of a morphism in the abelian category $\QCoh(X)$ into a surjection followed by an injection is unique, we obtain the desired identification of the two surjections displayed above.

\end{proof}

By intersecting with the preimage of $z\in Z$, we obtain the following immediate consequence. 

 \begin{corollary}\label{cor:psi_dagger_fiber_B}
 Let $X \xrightarrow{\psi} Y \xrightarrow{\phi} Z$ be morphisms {of algebraic varieties}. If $\psi$ is a submersion then 
 \[
 (\phi\circ\psi  \circ\Pi_{\Bi_{\psi\circ \phi}})^{-1}(z) = \psi^\dagger \left( (\phi \circ\Pi_{\Bi_{\phi}})^{-1}(z) \right).
 \]
 \end{corollary}
{We are ready to prove our descent result for} {quasi-transitivity.}
\begin{prop}[Descent for Quasi-Transitive Morphisms]
\label{prop:descent} 
 Let $X \xrightarrow{\psi} Y \xrightarrow{\phi} Z$ be morphisms {of algebraic varieties}, where $\psi$ is a submersion. 
 \begin{enumerate}
 \item If $\phi$ is quasi-transitive then so is $\phi\circ \psi$. 
 \item Conversely, if $\psi$ is surjective and $\phi\circ \psi$ is quasi-transitive then so is $\phi$. 
 \end{enumerate} 
\end{prop}

\begin{proof}
For $(1)$, let $x\in X$ and let $(x,v)\in \Bi_{\phi\circ \psi}$. We have to show that 
\[
\dim_{(x,v)}(\phi\circ \psi \circ \Pi_{\Bi_{\phi\circ \psi}})^{-1}(\phi\circ \psi(x)) = \dim_{x} X.  
\]
By \Cref{cor:psi_dagger_fiber_B}, the left-hand side is 
$
\dim_{(x,v)}(\psi^\dagger(\phi\circ \Pi_{\Bi_\phi})^{-1}(\phi(\psi(x))))
$
and we may write $v = D\psi^*(u)$ for some $u\in T^*_{\psi(x)} Y \otimes_{k(\psi(x))} k(x)$. 
We deduce that
\begin{align*}
\dim_{(x,v)}(\phi\circ \psi \circ \Pi_{\Bi_{\phi\circ \psi}})^{-1}(\phi\circ \psi(x)) & \stackrel{(\text{lem. }\ref{cor:psi_dagger_fiber_B})}{=} \dim_{(x,v)}(\psi^\dagger(\phi\circ \Pi_{\Bi_\phi})^{-1}(\phi(\psi(x)))) \stackrel{(\text{cor. } \ref{lem:dagger_dim})}{=} \\
&=
\dim_{(\psi(x),u)}((\phi\circ \Pi_{\Bi_\phi})^{-1}(\phi(\psi(x)))) + \dim_x X - \dim_y Y \stackrel{(*)}{\le} \\
&\le \dim_y(Y)  + \dim_x X - \dim_y Y= \dim_x X.
\end{align*}
where $(*)$ follows from the assumption that $\phi$ is quasi-transitive.

For $(2)$, given a point $(y,u)\in \Bi_{\psi}$, using \Cref{prop:pullback_for_Bi} and the surjectivity of $\psi$, we can find a point $(x,v)\in \Bi_{\phi\circ \psi}$ which corresponds to it under the correspondence $\psi^\dagger$. The result now follows by a similar computation to the one in the previous item.

\end{proof}

\subsection{Applications of Artin's approximation theorem}
For an algebraic variety $Y$ and a point $y\in Y$, we let $Y_y = \Spec(\cO_{Y,y})$ be the localization of $Y$ at $y$, $ \widehat{Y}_y = \Spec(\widehat{\cO}_{Y,y})$ the completion of $Y_y$ and $Y_y^h$ the Henselization of $Y_y$. For a natural number $n$, we further denote  by $Y_{y,n}:= \Spec(\cO_{Y,y}/\mathfrak{m}_y^n)$ where $\mathfrak{m}_y$ is the maximal ideal of $\cO_{Y,y}$. 
We record the following formulation of the Artin Approximation Theorem:
\begin{thm}[Artin Approximation] \label{prop:artin_section}
Let $\phi \colon X\to Y$ be a morphism of algebraic varieties.
Assume that there is a section 
from the formal completion of $Y$ at a point $y$ to $X$:

\[
\xymatrix{
                               & & X\ar^{\phi}[d] \\
\widehat{Y}_y\ar^{\hat{s}}[rru]\ar[rr]&   & Y 
}
\]
Then, for every $n \in \mathbb{N}$ there exists a morphism $s^h\colon Y_y^h\to X$ making the following diagram commutative: 
\[
\xymatrix{
 \widehat{Y}_{y}\ar^{\hat{s}}[rr]                                 &    & X\ar^{\phi}[d] \\
Y_{y,n}\ar^{\hat{s}|_{Y_{y,n}}}[rru]\ar[r]\ar[u] & Y^h_y\ar@{..>}_{s^h}[ru]\ar[r]      & Y 
}
\]

\end{thm}

\begin{proof}
Let $A = \cO(Y^h_y)$ be the Henselian local ring of $Y$ at $y$, $\mathfrak{m}_A$ its maximal ideal, and $\widehat{A}$ its completion at $\mathfrak{m}_A$ so that $\Spec(\widehat{A}) = \widehat{Y}_y$. 
Consider the functor 
\[
F\colon \{A\text{-algebras}\} \to \text{Sets}
\]
sending $A\to R$ to the set of lifts of the map $\Spec(R)\to Y^h_y \to Y$ along $\phi$ as in the diagram
\[
\xymatrix{
&    & X\ar^{\phi}[d] \\
\Spec(R)\ar@{..>}[rru]\ar[r] & Y^h_y\ar[r]      & Y 
}
\]
Since $X$ and $Y$ are of finite type, this functor is locally of finite presentation in the sense of \cite[Definition I.5]{artin1969algebraic}. The map $\hat{s}$ gives a point $\hat{s}\in F(\widehat{A})$. Hence, by \cite[Theorem I.12]{artin1969algebraic}, there is a point $s^h\in F(A)$ whose image in $F(A/\mathfrak{m}_A^n)$ coincides with that of $\hat{s}$, as required.
\end{proof}

\begin{cor}[Spread-out]\label{cor:spred}
Let $\phi \colon X\to Y$ be a morphism of algebraic varieties.
Assume that there is a section 
from the formal completion of $Y$ at a point $y$ to $X$:

\[
\xymatrix{
                               & & X\ar^{\phi}[d] \\
\widehat{Y}_y\ar^{\hat{s}}[rru]\ar[rr]&   & Y 
}
\]
Then, for every $n\in \mathbb{N}$, there is an \'{e}tale map $ U\to Y$ which is a Nisnevich neighborhood of $y$ and a map $s^U\colon U\to X$ that fits into a commutative diagram 
\[
\xymatrix{
 \widehat{Y}_{y}\ar^{\hat{s}}[rr]                                 &    & X\ar^{\phi}[d] \\
Y_{y,n}\ar^{\hat{s}|_{Y_{y,n}}}[rru]\ar@{..>}[r]\ar[u] \ar@/^-0.8pc/[rr] & U\ar@{..>}_{s^U}[ru]\ar[r]      & Y 
}
\]
\end{cor}

\begin{proof}  
By Artin Approximation, we have a lift $s^h\colon Y^h_y \to X$ which agrees with $\hat{s}$ on $Y_{y,n}$. 
{Since $X$ and $Y$ are of finite type,}
we can find a factorization of $s^h$ through such an \'{e}tale $Y$-scheme $U$, as required.
\end{proof}


%

\subsection{Structure Theorem}

To formulate the structure theorem for strong Thom stratifications we recall the following notion. 
\begin{defn}[Dimension of Morphism]
Let $\phi \colon X\to Y$ be a morphism of algebraic varieties.
We denote by $\dim(\phi)$ the maximum of the dimensions of the fibers of $\phi$.  
\end{defn}

Let $V$ be a finite dimensional vector space over $\C$, and let $X$ be an analytic variety over $\C$. Let $\xi \colon V\to \Gamma(X,\cT_X)$ be a linear map from $V$ to the space of analytic vector fields on $X$. We can then (locally) integrate the vector fields in $\xi(V)$ to get a continuous map $B \times U  \to X$ for a small disc $B$ around the origin of $V$ and a small open set $U\subseteq X$. Our next goal is to show that this construction works over an arbitrary field of characteristic zero by replacing the small neighborhoods in $V$ and $X$ by the formal completions.   

\begin{defn}
Let $F$ be a field of characteristic zero, let $V$ be a finite dimensional vector space over $F$, and let $\xi \colon V\to \Gamma(X,\cT_X)$ be a linear map. For a point $x\in X$ 
we define the map 
\[
\exp_\xi \colon \widehat{V\times X}_{(0,x)} \to \widehat{X}_x
\]
by the formula 
\[
\exp_\xi^*(f) = \sum_{n=0}^\infty \frac{(\sum_{i=1}^k t_i \xi(v_i))^n (f)}{n!},
\]
where:
\begin{itemize}
\item $v_1,\dots,v_k$ is a basis of $V$. 
\item $t_1,\dots,t_k$ the dual basis of $V^*$, which we regard as elements of $\widehat{\cO}_{V,0}$. 
\item We regard $\xi(v_i)$ as a derivation of $\widehat{\cO}_{X,x}$.
\end{itemize}
\end{defn}
The formula is clearly independent of the choice of basis, and it is straightforward to check that it defines a (continuous) ring map 
$
\exp_\xi^* \colon \widehat{\cO}_{X,x} \to \widehat{\cO}_{V\times X,(0,x)},
$
and hence a map 
$\exp_\xi\colon \widehat{V\times X}_{(0,x)} \to \widehat{X}_x$ as claimed.

\begin{prop}\label{prop:exp.props}
The maps $\exp_\xi$ have the following properties: 
\begin{enumerate}
\item\label{prop:exp.props:1} If $\xi = 0$ then $\exp_\xi$ is the completion of the projection $V\times X \to X$.  
\item The restriction of $\exp_\xi$ along the zero section $\widehat{X}_x\into \widehat{V\times X}_{(0,x)}$ is the identity of $\widehat{X}_x$.  
\item\label{prop:exp.props:3} The differential $D\exp_{\xi} \colon V\oplus T_x X \to T_x X$ is given by 
\[
D\exp_{\xi}(v,u)= \xi(v)|_x + u. 
\]
\item\label{prop:exp.props:4} Let $\phi \colon X\to Y$ be a morphism sending $x\in X$ to $y\in Y$, and let $\xi_X\colon V\to \Gamma(X,\cT_X)$ and $\xi_Y \colon V\to \Gamma(Y,\cT_Y)$ be compatible maps, in the sense that the diagram 
\[
\xymatrix{
V\ar^{\xi_Y}[d] \ar^{\xi_X}[r] & \Gamma(X,\cT_X)\ar^{D\phi}[d] \\
\Gamma(Y,\cT_Y)\ar^{\phi^*}[r] & \Gamma(X,\phi^*\cT_Y)
}
\]
commutes. Then the diagram 
\[
\xymatrix{
\widehat{V\times X}_{(0,x)}\ar^{\exp_{\xi_X}}[r]\ar_{\widehat{\id_V \times \phi}}[d] & \widehat{X}_x \ar^{\hat{\phi}}[d] \\
\widehat{V\times Y}_{(0,y)} \ar^{\exp_{\xi_Y}}[r] & \widehat{Y}_y
}
\]
also commutes.
\end{enumerate}
\end{prop}

\begin{proof}
All these claims follow easily from the formula of $\exp_\xi$. Specifically:
\begin{enumerate}
\item If $\xi(v)(f) \equiv 0$ for all $v\in V$ then $\exp_\xi^*(f) = f$ for all $f\in \widehat{\cO}_{X,x}$.
\item Follows from the fact that $\exp_\xi^*(f) \equiv f \mod (t_1,\dots,t_n)$. 
\item Follows from the fact that $\exp_\xi^*f   = f + \sum_i t_i \xi(v_i)(f) \mod (t_1,\dots,t_k)^2$. 
\item Follows from the fact that, by our compatibility assumption between $\xi_X$ and $\xi_Y$, we have 
\[
\xi_X(v)(\phi^*f) = \phi^*(\xi_Y(v)f)
\]
for all $f\in \widehat{\cO}_{Y,y}$ and all $v\in V$.
\end{enumerate}
\end{proof}

\begin{lem}\label{lem:form.ODE}
Let $\phi \colon X\to Y$ be a strongly Thom stratified morphism {of smooth algebraic varieties cover a field $F$ of characteristic $0$.} and let $S\subset X$ be a stratum. Let $i\colon Z\into X$ be a locally closed embedding of a smooth subvariety and let $x\in S\cap Z$ such that $$T_x(S\cap\phi^{-1}(\phi(x)))\oplus T_xZ=T_xX.$$  
Then there exists a morphism $\psi \colon \widehat X_x \to Z$ that renders the following diagram commutative:
\[
\xymatrix{
\{x\} \ar@{=}[d]\ar[r] &\widehat{X}_x\ar[r]\ar^\psi@{..>}[d] & X \ar^\phi[r] & Y \\
\{x\}\ar[r]&Z\ar_i[rr]                     &                 & X\ar^{\phi}[u]                 
}
\]
and such that the induced map on tangent spaces $D\psi_x \colon T_{x}X \cong T_x \widehat{X}_x \to T_{\psi(x)} Z$ is surjective. 
\end{lem}
\begin{proof} 
Let $v_1,\dots,v_k\in \Ker(D\psi)$ be independent vector fields defined in a neighborhood of the point $x$ which span the tangent space $T_x S\cap \phi^{-1}(\phi(x))$ (such exist by the assumption that $\phi$ is strongly Thom stratified). By shrinking $X$ if necessary, we may assume that the $v_i$’s are defined on all of $X$. Let $V$ be the span of the $v_k$’s in $\Gamma(X,\cT_X)$ and let $\xi \colon V\into \Gamma(X,\cT_X)$ be the inclusion. We can then form the map $\exp_\xi \colon \widehat{V\times X}_{(0,x)} \to \widehat{X}_x$. 

Let $i\colon Z\into X$ be the inclusion, and consider the composition 
\[
\nu \colon \widehat{V\times Z}_{(0,x)}  \oto{\widehat{\id_V \times i}} \widehat{V\times X}_{(0,x)} \oto{\exp_\xi} \widehat{X}_x.
\]
We claim that $\nu$ is an isomorphism. Indeed, since $\nu^*$ is a local morphism between the formal completions of smooth algebras over $F$, it suffices to show that $\nu$ induces an isomorphism on the tangent spaces. By \Cref{prop:exp.props}\eqref{prop:exp.props:3}, the map $D\nu$ is identified with the sum map 
\[
T_x S \cap \phi^{-1}(\phi(x)) \oplus   T_x Z \to T_x X,
\]
which is an isomorphism by our assumption. 

We now define the map $\psi$ by the composition 
\[
\psi\colon \widehat{X}_x \oto{\nu^{-1}} \widehat{V\times Z}_{(0,x)} \oto{\text{pr}_Z} \widehat{Z}_{x} \to Z. 
\]
The map $D\psi$ is surjective since $D\text{pr}_Z$ is a linear projection. 
The diagram in the statement of the proposition is the outer diagram of:
\[
\begin{tikzcd}[column sep=large, row sep=large]
\widehat{V\times Z}_{(0,x)} 
  \arrow[r, "\widehat{\id_V\times i}"] 
  \arrow[d, "pr_Z"'] 
  \arrow[rr, bend left=25, "\sim"{below}, "\nu"] 
  & \widehat{V\times X}_{(0,x)} 
    \arrow[r, "\exp_\xi"] 
    \arrow[d, "pr_X"'] 
  & \widehat{X}_x 
    \arrow[r] 
   \arrow[rounded corners, to path={ -- ([yshift=7ex]\tikztostart.north) -- 
    node[above]{$\psi$}
    ([xshift=-45ex, yshift=7ex]\tikztostart.north) -- 
    ([xshift=-4.7ex, yshift=-0ex]\tikztotarget.west) -- 
    (\tikztotarget)}]{dll}
  & X 
    \arrow[d, "\phi"] \\
\widehat{Z} 
  \arrow[r, "\hat{i}"] 
  \arrow[rrd] 
  & \widehat{X}_x 
    \arrow[r] 
  & X 
    \arrow[r, "\phi"] 
  & Y \\
& & Z 
  \arrow[u, "i"] 
\end{tikzcd}
\]
so it is enough to prove that the latter is commutative.

The left square and the bottom triangle clearly commute. To see that the right rectangle commutes, note that since $v_i\in \Ker(D\phi)$, the morphism $\xi \colon V\to \Gamma(X,\cT_X)$ is compatible with the zero morphism $0\colon V\to \Gamma(Y,\cT_Y)$. Hence, the commutativity of the rectangle follows from 
items \eqref{prop:exp.props:1} and \eqref{prop:exp.props:4} of \Cref{prop:exp.props}. The rest of the faces in the diagram commute by the definitions. Hence, the entire diagram commutes.
\end{proof}
\begin{lemma} \label{lem:transversal} Let $\phi :X \rightarrow Y$ be a strongly Thom stratified {morphism}, let $S \subseteq X$ be a stratum, let $x\in S$, and let $i:Z \rightarrow X$ be a transversal to $S$ at $x$. Then there is a Zariski neighborhood $U$ of $x$ such that $\phi \circ i|_{i ^{-1}(U)}$ is regular.
\end{lemma} 
\begin{proof} 
{Let $\cT_\phi^\str$ be the subsheaf of $\cT_\phi$ spanned by the vector fields that are tangent to the strata of $X$.}
Let $v_1,\ldots,v_k$ be a basis of $T_x(S\cap \phi ^{-1} (\phi(x)))$. By the assumption, there is a neighborhood $U$ of $x$ and sections $s_1,\ldots,s_k\in \mathcal{T}_{\phi}^{\str}(U)$ such that $s_i(x)=v_i$. Shrinking $U$ if needed, we can assume that $s_1(z),\ldots,s_k(z),T_zZ$ span $T_zX$ for every $z\in Z\cap U$. 

Let $z\in Z\cap U$ and denote the stratum containing $z$ by $S_z$. Since $s_i(z)$ is tangent to the fibers and the strata, we get
\[
T_z (S_z \cap \phi^{-1}\phi(z)) + T_z Z = T_z X,
\]
proving regularity.
\end{proof} 

We turn to prove the main theorem of the section.
\begin{thm} \label{thm:structure.Thom}
Let $\phi \colon X\to Y$ be a strongly Thom stratified morphism {of smooth algebraic varieties over a field of characteristic $0$,} and let $S\subset X$ be a closed stratum such that $\dim(\phi|_S)> 0$. Then, there exists a commutative diagram 
\[
\xymatrix{
\widetilde{X}\ar^j[r]\ar_\pi[d] & X \ar^\phi[r] & Y \\
Z\ar_{\widetilde{\phi}}[rru] &       &               
}
\]
where
\begin{enumerate}
\item \label{cond:pi.subm} $\pi$ is a surjective submersion. 
\item \label{cond:j.Nis} $j$ is an \'{e}tale morphism and a Nisnevich cover of a Zariski neighborhood of $S$.   
\item \label{cond:dimZ} $\dim(Z)< \dim(X)$.
\end{enumerate}
Moreover, $\tilde{\phi}$ is quasi-transitive.
\end{thm}

\begin{proof}
Since every Nisnevich cover of an algebraic variety has a finite subcover, it is enough to prove that for every $x\in S$ there is a commutative diagram of algebraic varieties
\begin{equation} \label{eq:diagram_local}
\xymatrix{
\widetilde{X}_x\ar^{j_x}[r]\ar_{\pi_x}[d] & X \ar^\phi[r] & Y \\
Z_x\ar_{\widetilde{\phi}_x}[rru]                     &                 &           
}
\end{equation}
satisfying conditions \eqref{cond:pi.subm},\eqref{cond:dimZ}, and
\begin{enumerate} 
\item[(\ref{cond:j.Nis}')] $j_x$ is an \'{e}tale morphism and a Nisnevich cover of a Zariski neighborhood of $x$.\\
\end{enumerate}

Choose a locally closed embedding $i'_x\colon Z'_x\into X$ of a smooth subvariety such that
\begin{itemize}
    \item $x\in S\cap Z'_x$ 
    \item  $T_x(S\cap\phi^{-1}(\phi(x)))\oplus T_xZ=T_xX.$  
\end{itemize}

By \Cref{lem:form.ODE}, there is a 
commutative diagram 
\[
\xymatrix{
\widehat{X}_x\ar^{}[r]\ar_{\pi'_x}[d] & X \ar^\phi[r] & Y \\
Z'_x\ar^{i'_x}[rr]                     &                 & X\ar^{\phi}[u]                
}
\]
in which $\widetilde{X}_x \rightarrow X$ is the inclusion of the completion of $X$ at $x$ and $\pi_x'$ induces a surjection on the tangent spaces at $x$. 
This diagram gives a map $\hat{s}:\widehat{X}_x\to  X\times_Y Z'_x$.
Consider the projections:
$$p_X:X\times_Y Z'_x\to X \quad \text{and} \quad p_{Z_x'}\colon X\times_Y Z'_x\to Z_x'$$

Applying \Cref{cor:spred} to 
\[
\xymatrix{
                               & & X\times_Y Z'_x\ar^{p_X}[d] \\
\widehat{X}_x\ar^{\hat{s}}[rru]\ar[rr]&   & X 
},
\]
there is an \'{e}tale map $j_x'\colon  U\to X$ which is a Nisnevich neighborhood of $x$ and a map $s^U\colon U\to X\times_Y Z'_x$ that fits into a commutative diagram:
\[
\xymatrix{
 \widehat{X}_{x}\ar^{\hat{s}}[rr]                                 &    & X\times_Y Z'_x\ar^{p_X}[d] \\
X_{x,1}\ar^{\hat{s}|_{X_{x,1}}}[rru]\ar@{..>}[r]\ar[u] \ar@/^-1.5pc/[rr] & U\ar@{..>}_{s^U}[ru]\ar_{j'_x}[r]      & X 
}.
\]
From this, we obtain the following commutative diagram 
\[
\xymatrix{
U\ar^{j'_x}[r]\ar_{\pi_x':=p_{Z_x'} \circ s^U}[d] & X \ar^\phi[r] & Y \\
Z'_x\ar^{i'_x}[rr]                     &                 & X\ar^{\phi}[u]                
}
\]

Finally, we now show that there are open 
subsets $\widetilde{X}_x \subseteq U$ and 
$Z_x\subseteq Z_x'$ such that 
$\pi_x'(\widetilde{X}_x)\subseteq Z_x$ and 
such that the outer diagram in the 
following commutative diagram satisfies 
conditions \eqref{cond:pi.subm}, (\ref{cond:j.Nis}'), and \eqref{cond:dimZ}: 

\[
\begin{tikzcd}[column sep=large, row sep=large]
\widetilde{X}_x \arrow[hookrightarrow]{r} 
  \arrow["j_x", bend left=30]{rr} 
  \arrow["\pi_x"']{d} 
  & U \arrow["j'_x"]{r} 
    \arrow["\pi'_x"']{d} 
    & X \arrow["\phi"]{r} 
      & Y \\
Z_x \arrow[hookrightarrow]{r} 
  \arrow[rounded corners, to path={ -- ([yshift=-3ex]\tikztostart.south) -- 
    node[below]{$\widetilde{\phi}_x$}
    ([xshift=44ex, yshift=-3ex]\tikztostart.south) -- 
    ([xshift=44ex, yshift=12ex]\tikztostart.south) -- 
    (\tikztotarget)}]{rrru}
  & Z'_x \arrow["i'_x"]{rr} 
  && X \arrow["\phi"]{u}
\end{tikzcd}
\]

Let $u\in U$ be the image of $x$ under the map $X_{x,1} \rightarrow U$. Since $\pi_x'\colon \widehat{X}_x \to Z_x'$ induces a surjection on tangent spaces, the composition $X_{x,1} \to U \to Z_x'$ induces a surjection on the tangent 
spaces at $x$, so $\pi_x'$ is a submersion at $u$. Let $\widetilde{X}_x$ be a Zariski neighborhood of $u$ for which the restriction 
$\pi_x'\restriction_{\widetilde{X}_x}$ 
is a submersion and let 
$Z_x:=\pi_x'(\widetilde{X}_x)$. 
{It is easy to check that the resulting diagram \eqref{eq:diagram_local} 
satisfies conditions \eqref{cond:pi.subm}, (\ref{cond:j.Nis}'), and 
\eqref{cond:dimZ} as required.}

Finally, to show that $\tilde{\phi}$ is quasi-transitive, note that by \Cref{prop:descent}(1) the morphism $\phi\circ j = \tilde{\phi}\circ \pi$ is quasi-transitive, and hence by \Cref{prop:descent}(2) the morphism $\tilde{\phi}$ is quasi-transitive. 

\end{proof}

\section{Direct images of smooth measures}
{Finally, we will use the geometric results from the previous section to prove our main result,  \Cref{thm:intro.main}. We start by recalling the standard notions related to measures on the $F$-points of smooth varieties.}

\begin{definition} Let $X$ be a 
smooth variety over a
{$p$-adic field $F$.}
Denote the ring of integers of $F$ by $O$. \begin{enumerate}
\item A (complex-valued) measure $\mu$ on $X(F)$ is called smooth if for every $x\in X(F)$ there is a neighborhood $x \in U \subseteq X(F)$ and a{n analytic}\footnote{{See e.g. \cite{Ser64} for the notion of $F$-analytic manifolds and maps between them.}} diffeomorphism $f:U \rightarrow O^d$ such that $f_* \mu$ is an additive Haar measure on $O^d$.
\item We denote the complex vector space of all compactly supported measures on $X(F)$ by $\mathcal{M}_c(X(F))$ and the complex vector space of compactly supported and smooth measures on $X(F)$ by $\mathcal{M}^{\infty}_c(X(F))$. We consider both vector spaces as modules over the {(non-unital)} ring $C_c^\infty(X(F))$ of locally constant and compactly supported functions on $X(F)$.
\item For a morphism $\phi:X \rightarrow Y$, the pushforward {of measures gives}
a linear map $\phi_*: \mathcal{M}_c^\infty(X(F)) \rightarrow \mathcal{M}_c(Y(F))$.
{
\item For $x\in X(F)$, define the stalk $C^\infty_c(X(F))_x$ at $x$ to be the following $C^\infty_c(X(F))$-module: it is the one-dimensional vector space $\mathbb{C}$ on which $f\in C^\infty_c(X(F))$ acts by multiplication by $f(x)$.
\item For a $C^\infty_c(X(F))$-module $M$ and $x\in X(F)$, the space of stalk of $M$ at $x$ is 
\[
M_x:=M \otimes_{C^\infty_c(X(F))} C^\infty_c(X(F))_x.
\]
}
{In other words, $M_x$ is the quotient of $M$ by the submodule consisting of elements $m$ such that $1_U\cdot m = 0$ for some open compact neighborhood $U$ of $x$.}
\end{enumerate} 
\end{definition} 
{
We are ready to state and prove our main theorem.
}
\begin{thm}\label{thm:main} If $\phi:X\to Y$ is a quasi-transitive morphism of smooth algebraic varieties defined over a 
{$p$-adic field $F$}
then the dimensions of stalks of $\phi_* \left( \mathcal{M}_c \left( X(F) \right) \right)$ are uniformly bounded: there is an integer $N(\phi)$ such that 
\[
\dim \left(\phi_*(\mathcal{M}_c(X(F)))\right)_y<N(\phi),
\]
for every $y\in Y(F)$.
\end{thm}

\begin{proof} For every quasi-transitive map $\phi':X' \rightarrow Y'$, let 
\[
T(\phi'):=\inf \left\{ |P| \mid \text{$(X' \rightarrow P,Y' \rightarrow Q)$ is a Thom stratification of $\phi'$}\right\}. 
\]

By Theorem \ref{thm:trans_iff_thom}, $T(\phi)<\infty$. We prove the claim by induction on the pair $\left( \dim(X),T(\phi)\right)$ in the lexicographic order.

Given $\phi:X \rightarrow Y$, choose a Thom stratification 
\[
\xymatrix{X \ar[r] \ar[d] & P \ar[d] \\ Y \ar[r] & Q}
\]
with $|P|=T(\phi)$ and choose a closed stratum $S \subseteq X$. We consider the following cases:

\begin{enumerate}[{Case} 1]
\item $\dim \phi \restriction_S =0$:\\

{In this case,} there is a number $n=n(\phi,S)$ such that $|\phi ^{-1} (y)\cap S| \leq n$, for all $y\in Y(F)$. Given $y\in Y(F)$, let $\phi ^{-1} (y)\cap S(F)=\left\{ x_1,\ldots,x_m \right\}$ and fix smooth measures $\eta_1,\ldots,\eta_m\in \mathcal{M}_c(X(F))$ such that \begin{enumerate}
\item $\eta_i$ does not vanish at $x_i$.
\item The supports of $\eta_i$ are pairwise disjoint.\\
\end{enumerate}
Now, 
\[
X(F)=(X(F) \smallsetminus S(F)) \cup (X(F) \smallsetminus \phi ^{-1}(y)) \cup (\bigcup_i \supp(\eta_i)
)\]
is an open cover {of $X(F)$}. {Consequently,} every {measure} $\omega\in \mathcal{M}_c^\infty(X(F))$ can be written as $\omega=\eta+\theta+\zeta$, where \begin{enumerate}
\item $\eta$ is a linear combination of $\eta_1,\ldots,\eta_m$.
\item $\theta\in \mathcal{M}_c^\infty \left( (X \smallsetminus S)(F) \right)$.
\item $\zeta\in \mathcal{M}_c^\infty(X(F))$ and $\supp(\zeta) \cap \phi ^{-1} (y)=\emptyset$.
\end{enumerate}

It follows that
\[
\phi_*(\omega)_y=\phi_*(\eta)_y+\phi_*(\theta)_y,
\]
so
\[
\Big( \phi_* \left( \mathcal{M}_c^\infty(X(F)) \right)\Big)_y \subseteq \linspan \left\{ \phi_*(\eta_1)_y ,\ldots,\phi_*(\eta_m)_y \right\} + \Big( \phi_* \left( \mathcal{M}_c^\infty \left( (X \smallsetminus S)(F) \right) \right) \Big)_y.
\]
Hence, the claim holds with $N(\phi)=n(\phi,S)+N(\phi \restriction_{X \smallsetminus S})$.

\item $\dim \phi \restriction_S >0$:\\

By Theorem \ref{thm:structure.Thom}, there is a commutative diagram
\[
\xymatrix{
\widetilde{X}\ar^j[r]\ar_\pi[d] & X \ar^\phi[r] & Y \\
Z\ar_{\widetilde{\phi}}[rru] &       &               
}
\]
where
\begin{enumerate}
\item $\pi$ is a surjective submersion. 
\item $j$ is an \'{e}tale morphism and a Nisnevich cover of a Zariski neighborhood $U$ of $S$.   
\item $\dim(Z)< \dim(X)$.
\item $\widetilde{\phi}$ is quasi-transitive.
\end{enumerate}

Since $\pi$ is a submersion, $\pi_*(\mathcal{M}_c^\infty (\widetilde{X}(F))) \subseteq \mathcal{M}_c^\infty (Z(F))$. Since $j$ is \'{e}tale and a Nisnevich cover of $U$, $j_*\mathcal{M}_c^\infty (\widetilde{X}(F)) = \mathcal{M}_c^\infty(U(F))$. Therefore,
\[
\phi_*(\mathcal{M}_c^\infty(U(F))) = (\phi \circ j)_* \mathcal{M}_c^\infty(\widetilde{X}(F))=(\widetilde{\phi} \circ \pi)_* \mathcal{M}_c^\infty(\widetilde{X}(F)) \subseteq \widetilde{\phi}_* \mathcal{M}_c^\infty(Z(F)),
\]
so
\begin{align*}
\phi_*\left( \mathcal{M}_c^\infty(X(F))\right) &= \phi_*(\mathcal{M}_c^\infty(U(F)))+\phi_*(\mathcal{M}_c^\infty((X \smallsetminus S)(F))) \\ &\subseteq 
\widetilde{\phi}_* \mathcal{M}_c^\infty(Z(F)) + \phi_*(\mathcal{M}_c^\infty((X \smallsetminus S)(F))),
\end{align*}
and the claim holds with $N(\phi)=N(\widetilde{\phi})+N(\phi \restriction_{X \smallsetminus S})$.

\end{enumerate} 
\end{proof}

\section{Examples}

The next two examples show that the criterion of  \Cref{thm:intro.main} is not necessary.

\begin{example} \label{exam:1dimage} Suppose that $\phi: \mathbb{A} ^n \rightarrow \mathbb{A} ^1$. 

One can deduce from the theory of $p$-adic integration (see e.g. \cite{CH18})
that the $C_c^\infty(F)$-module $\phi_* \left( \mathcal{M}_c(F^n) \right)$ has finite-dimensional stalks.
\end{example} 

\begin{example} \label{exam:4lines} We give an example of a map $\phi:\mathbb{A}^3 \rightarrow \mathbb{A} ^1$ that is not quasi-transitive but for which $\phi_* \left( \mathcal{M}_c(\Q_p^3) \right)$ has finite-dimensional stalks. 

Consider the map $\phi:\mathbb{A}^3 \rightarrow \mathbb{A} ^1$ given by $\phi(x,y,z)=xy(x+y)(x+yz)$. 
 By Example \ref{exam:1dimage}, $\phi_*$ has finite-dimensional stalks.

 In order to show that $\phi$ is not quasi-transitive, we may replace the base field by $\C$.
 
We claim that, for every $z_0\neq 0,1$, there is no vertically extendible nonzero tangent vector in $T_{(0,0,z_0)}\mathbb{A} ^3$. Indeed, assume that $U$ is a small ball around $(0,0,z_0)$ in $\mathbb{C} ^3$ and $\xi \in \mathcal{T}_\phi (U)$ with $\xi(0,0,z_0)\neq0$. 
\begin{enumerate}
\item If $U$ is small enough, then the only singular points of $\phi$ in $U$ are points of the form $(0,0,t)$. Thus, $\xi(0,0,z_0)$ is a multiple of $(0,0,1)$.
\item After normalizing $\xi$ so that its third coordinate is 1, the flow of $\xi$ gives rise to a diffeomorphism {$\Phi$} between two neighborhoods of the origin in $\mathbb{C}^2$ that carries the plane curve $\{xy(x+y)(x+z_0y)=0\}$ to the plane curve $\{xy(x+y)(x+z_0'y)=0\}$ for some $z_0'\neq z_0$. Looking at the derivative of {$\Phi$}, it is clear that such a diffeomorphism cannot exist.
\end{enumerate} 
It follows that $\dim 
\Pi_{\mathfrak{B}_\phi} ^{-1}(0,0,z_0)
=3$, so $(\phi \circ \Pi_{\mathfrak{B}_\phi}) ^{-1} (0)$ is 4-dimensional, and $\phi$ is not quasi-transitive. On the other hand, by Example \ref{exam:1dimage}, $\phi_*$ has finite-dimensional stalks.
\end{example}  
{
\begin{remark}\label{rem:def}
 A direct computation shows that the map $(x,y,z)\mapsto x^2y(x+y)$ is quasi-transitive. 
 
 This means that we have a 1-parameter family of morphisms $\phi_\alpha:\A^3\to\A^1$, given by $\phi_\alpha(x,y,z)=xy(x+y)(x+\alpha yz)$, such that for all $\alpha\neq 0$ the morphism $\phi_\alpha$ is not quasi-transitive near the origin in the source, but for $\alpha=0$ the morphism $\phi_\alpha$ is quasi-transitive.
\end{remark}
}
The following is an example of a map $\phi : X \rightarrow Y$ for which $\phi_*(\mathcal{M}_c(X(F)))$ has infinite-dimensional stalks.

\begin{example} \label{exam:blowup} 
{Let $F$ be a $p$-adic field.}
Let $\pi : \Bl_0 \mathbb{A} ^2 \rightarrow \mathbb{A} ^2$ be the blowup of the plane at the origin and let $E = \pi ^{-1} (0)$ be the exceptional divisor. $E$ is identified with the projective line $\mathbb{P} ^1$. Let $d$ be a metric on $\mathbb{P}^1(F)$ that is compatible with the topology.

Let $p\in E{(F)}$ and let $\epsilon >0$. If $U \subseteq \Bl_0 \mathbb{A} ^2{(F)}$ is a small enough neighborhood of $p$, then $\pi(U)$ is contained in the sector
\[
\left\{ (x,y)\in \mathbb{A}^2{(F)} \smallsetminus 0 \mid d([x:y],p)< \epsilon \right\} \cup \left\{ 0 \right\}.
\]
By shrinking $\epsilon$ we get that the space of germs at the origin of supports of measures in $\phi_*(\mathcal{M}_c(\Bl_0{\mathbb{A}}^2(F)))$, i.e. the direct limit
\[
\colim_{U} \left\{ \supp(\mu)\cap U \mid  \mu \in \phi_*(\mathcal{M}_c^\infty(\Bl_0(F))) \right\} 
\]
is infinite. In particular, $\phi_*(\mathcal{M}_c(\Bl_0{\mathbb{A}^2}(F)))$ has an infinite-dimensional stalk at the origin.
\end{example} 

More generally,
\begin{proposition}\label{prop:flat} Let $X,Y$ be smooth varieties over 
a {$p$-adic field $F$} 
and let $\phi :X \rightarrow Y$ be a dominant morphism. Assume that there is a point $y\in Y(F)$ such that $\dim \phi ^{-1} (y)(F)>\dim X -\dim Y$.\footnote{Here we can take, for example, the notion of dimension of semi-algebraic sets as in \cite[\S 3]{SvdD}} Then the space of germs of supports of pushforwards
\[
\colim_{y\in U} \left\{ \supp(\mu)\cap U \mid \mu \in \phi_*(\mathcal{M}_c^\infty(X(F))) \right\} 
\]
is infinite.
\end{proposition} 
For the proof we will need the following standard lemma.
\begin{lem}\label{lem:curv.sel}
    Let $X$ be a smooth algebraic variety defined over an infinite field $F$. Let $U\subset X$ be an open dense subset. Let $x\in X(F)\setminus U(F)$. Then there exists a smooth curve $S\subset X$ such that $x\in S(F)$ and $S\setminus x\subset U$.        
\end{lem}
\begin{proof}
    Replacing $X$ by a neighborhood of $x$ we can assume that we have an \'{e}tale map $\phi:X\to \bA^n$. Let $L\subset \A^n$ be an $F$-line passing through $\phi(x)$ such that $L \not \subseteq \phi(X\smallsetminus U)$. 

    Now take $S:=\phi^{-1}(L)\smallsetminus(Z\smallsetminus x)$.
\end{proof}
\begin{proof}[Proof of \Cref{prop:flat}]
We call an algebraic curve $C\subset Y$ nice if $\phi$ is flat over $C \smallsetminus \{y\}$. Let 
$$ Z_C:=\overline{\phi^{-1}(C\smallsetminus \{x\})}\cap \phi^{-1}(x)$$ 
and 
$$  \aZ_C:=\overline{\phi^{-1}(C\smallsetminus \{x\})(F)}^{\mathrm{an}}\cap \phi^{-1}(x)(F).$$ 
where $\overline{(-)}^{\mathrm{an}}$ denotes the closure in the analytic topology.
The locus $Y^{reg}$ of the regular values of $\phi$ is Zariski dense in $Y$. Thus,  
since $\phi$ is locally dominant, the locus
$X^{reg}:=\phi^{-1}(Y^{reg})$ is dense in $X$.

Let $x\in \phi^{-1}(y)(F)$. By \Cref{lem:curv.sel}
we have a smooth irreducible curve $S$ passing through $x$ such that $S\smallsetminus \{x\}\subset X^{reg}$.
Thus $C:=\phi(S)$ is good, and $x\in \mathscr Z_C$.

We obtain that the collection
\[
\left\{ \mathscr Z_C \mid \text{ $C$ is a nice curve passing through $x$} \right\}
\]
covers $\phi^{-1}(y)(F)$.

Each of the schemes $Z_C$ is of smaller dimension than $\phi^{-1}(y)$. Therefore, each of the analytic varieties $\aZ_C$ is of smaller dimension than $\phi^{-1}(y)$, so any finite union of analytic varieties of the type $\aZ_C$ will not cover $\phi^{-1}(y)(F)$. It follows that, for every $n$, there are nice curves $\bfC_1,\ldots,\mathbf{C}_n$ that pass through $x$ such that 
\[
\aZ_{{C}_i} \not\subset \bigcup_{j < i} \aZ_{{C}_j}.
\]
Choosing $x_i\in \aZ_{C_i}\smallsetminus\left(\bigcup_{j< i} Z_{\aC_j}\right)$, the images of small enough neighborhoods of $x_i$ have distinct germs, as required.
\end{proof}
\begin{cor}
Let $X,Y$ be connected smooth varieties over a 
{$p$-adic field $F$}
and let $\phi :X \rightarrow Y$ be a non-flat dominant morphism. 

Then there is a finite extension $E/F$ such that 
$\phi_*(\mathcal{M}_c(X(E)))$ has infinite-dimensional stalks.
\end{cor}
\begin{proof}
Choose an extension $E/F$ such that there is $y\in Y(E)$ satisfying  $\dim(\phi^{-1}(y))= \dim\phi^{-1}(y)(E)$. Now the statement follows from \Cref{prop:flat}.
\end{proof}

However, the converse does not always hold. Specifically,
the following gives an example of a flat map $\psi:X\to Y$ for which 
$\psi_*\mathcal{M}_c(X(F))$ has infinite-dimensional stalks.
\begin{prop}
    Consider the map $\phi:\A^2\to \A^2$ given by $\phi(x,y)=(x,xy)$ and define $\psi:\A^4\to \A^2$ by $\psi(x,y,z,w)=\phi(x,y)+\phi(z,w)$.
    Then $\psi$ is flat and $\psi_*(\mathcal{M}_c(\Q_p^4))$ has an infinite-dimensional stalk at $0$.
\end{prop}

\begin{proof}
  Let $\lambda_k$ be the normalized Haar measure on $\mathbb{Z}_p^k$ and let $\mu_n=\psi_*\left( \lambda_2 \cdot 1_{p^n \mathbb{Z}_p^2} \right)$. Note that 
$$\psi_*\left( \lambda_4 \cdot 1_{p^n \mathbb{Z}_p^4} \right)=\mu_n*\mu_n,$$ where $*$ denotes the convolution.

Fix an additive character $\chi$ of $\Q_p$ which is trivial on $\Z_p$ and non-trivial on $p^{-1}\Z_p$.
The Fourier transform of $\mu_n$ (w.r.t. the character $\chi$) is given by
\begin{align*}
\widehat{\mu_n}(a,b)
&= \int_{p^n \mathbb{Z}_p^2} \chi(ax + bxy)\, dx\, dy 
= \int_{p^n \mathbb{Z}_p} \chi(ax) 
   \left( \int_{p^n \mathbb{Z}_p} \chi(bxy)\, dy \right) dx \\
&= p^{-n} \int_{p^n \mathbb{Z}_p} \chi(ax)\,
   1_{\{x : |bx| \le p^{-n}\}}(x)\, dx 
= p^{-n} \int_{p^n \mathbb{Z}_p \cap p^n b \mathbb{Z}_p} \chi(ax)\, dx \\
&=
\begin{cases}
    0, & \text{if } |a| > p^n \max(|b|,1),\\[6pt]
    \dfrac{p^{-2n}}{\max(|b|,1)}, & \text{if } |a| \le p^n \max(|b|,1).
\end{cases}
\end{align*}

Therefore, for any $N\in\N$,
$$\supp(\widehat{(\mu_n*\mu_n)\cdot 1_{p^N\Z_p^2}})=
\supp(\widehat{\mu_n}^2 * 1_{p^{-N}\Z_p^2})=p^{-N}\Z_p^2 \cup \{(a,b)\in\Q_p^2\mid \valu(a) \ge n+ \min(\valu(b),0)\}.
$$
Therefore, for any $N\in\N$, the sequence $n\mapsto (\mu_n*\mu_n)\cdot 1_{p^N\Z_p^2}$ is linearly independent. Thus the sequence of stalks $(\mu_n*\mu_n)_0\in \psi_*(\mathcal{M}_c(\Q_p^4))_0$ is linearly independent.
\end{proof}
\begin{remark}
In view of \Cref{thm:intro.main} this implies that $\psi$ is not quasi-transitive.

We can also show this directly. The differential of $\psi$ is
\[
D\psi_{x,y,z,w}=(dx+dz,ydx+xdy+wdz+zdw).
\]
Thus $\Ker(D\psi)$ is the subsheaf of $\cT_{\A^4}$ generated by the vector fields 
\begin{align*}
  v_1:=&-x\partial_x+(y-w)\partial_y+x\partial_z,\\
  v_2:=&-z\partial_x+z\partial_z+(y-w)\partial_w,\\ 
  v_3:=&-z\partial_y+x\partial_w.  
\end{align*}
Each $v_i$ defines a regular function on $T^*\A^4$.
The variety $\fB_\psi$ is the zero-locus of these functions.
{Note that all the $v_i$-s vanish when $x=z=0$ and $y=w$}. Hence, 
\[
\{(x,y,z,w,\alpha dx+\beta dy+\gamma dz+\delta dw)\mid x=z=0,\ y=w\}\subset (\psi \circ \Pi_{\mathfrak{B}_\psi}) ^{-1} (0,0)
\]
and $\dim (\psi \circ \Pi_{\mathfrak{B}_\psi}) ^{-1} (0,0)>4$, {so that $\psi$ is not quasi-transitive.}

\end{remark}

{
\begin{remark}\label{rem:con}
More generally, the morphism $\psi_n:\A^{2n}\to \A^2$ given by $$\psi(x_1,y_1,\dots, x_n,y_n)=\phi(x_1,y_1)+\cdots \phi(x_n,y_n)$$ is neither {quasi-transitive} nor \FS.

$\psi_n$ is the self-convolution (in the sense of \cite[Definition 1.1]{GH2}) of $n$ copies of $\phi$. So, morphisms might not become {quasi-transitive} or \FS\ after any number of self-convolutions. This is 
unlike other properties of morphisms (see e.g. the main results of \cite{GH2}).

Note also that $\psi_n$ is a quadratic polynomial morphism whose strength (in the sense of \cite[Definition 1.4.]{Erm19}) tends to infinity. The failure of $\varphi_n$ to be quasi-transitive or $\mathcal{M}$-finite is in contrast to the Ananyan–Hochster principle described in \cite[\S 1.II]{Erm18}.
\end{remark}
}    

\bibliographystyle{amsalpha}

\bibliography{push}

\end{document}